\documentclass[10pt]{article}
\usepackage[utf8]{inputenc}

\usepackage{amsthm}
\usepackage{amsfonts} 
\usepackage{amsmath}
\usepackage{mathtools}
\usepackage{geometry}
\usepackage{graphicx,amsmath,fullpage,mathptmx,helvet}
\usepackage[ruled, lined, linesnumbered, commentsnumbered, longend]{algorithm2e}
\usepackage{stmaryrd}
\usepackage{tikz}
\usepackage{tikz-qtree}
\usepackage{listings}
\usepackage{caption}
\usepackage{subcaption}
\usepackage{chngcntr}
\usepackage{authblk}
\usepackage{enumerate}
\usepackage{url}
\usepackage{mathrsfs}
\usepackage{natbib}
\usepackage{pgfplots}

\usepackage[hidelinks]{hyperref}
\usepackage{todonotes}
\usepackage{cleveref}

\pgfplotsset{compat=1.18}
\theoremstyle{definition}

\theoremstyle{plain}
\newtheorem{theorem}{Theorem}[section]
\newtheorem{lemma}[theorem]{Lemma}

\newtheorem{corollary}[theorem]{Corollary}

\theoremstyle{remark}

\makeatletter
\@ifdefinable\skpt{\def\skpt#1\skpt{}}{}
\makeatother

\author{Antoine Lhomme, Nicolas Catusse, Nadia Brauner\\
Univ. Grenoble Alpes, CNRS, Grenoble INP, G-SCOP, 38000 Grenoble, France}

\title{On the convergence of computational methods\\ for the online bin stretching problem}

\begin{document}
\maketitle

\begin{abstract}
    Online bin stretching is an online packing problem where some of the best known lower and upper bounds were found through computational searches. The limiting factor in obtaining better bounds with such methods is the computational time allowed. However, there is still no theoretical guarantee that such methods do converge towards the optimal online performance. This paper shows that such methods do, in fact, converge; moreover, bounds on the gap to the optimal are also given. These results frame a theoretical foundation for the convergence of computational approaches for online problems.
\end{abstract}

\section{Online bin stretching and computational methods}

The online bin stretching problem, introduced in \cite{AZAR200117}, is defined as follows: a finite sequence of items, each characterized by its size (represented by a real number), must be placed into $m$ bins. The items are revealed sequentially; an item, when revealed, must be placed irremediably into a bin before the next one is shown. The whole set of items given is also known to fit inside $m$ bins of unit size.
The objective is to place all the items into the bins so that, in the end, the load of the fullest bin is minimized. An algorithm that places items into bins has its performance measured through its worst-case input; the performance of an algorithm is called in this problem the \textit{stretching factor}\footnote{If the sequence of items fits \textit{exactly} into $m$ bins of unit size, then the stretching factor may also be seen as a competitive ratio.}. This problem may also be interpreted as an online scheduling problem on $m$ parallel identical machines where the objective is to minimize the makespan (the total time to execute all tasks): using scheduling notations, the problem may be written as $P_m|online, OPT=1|C_{max}$, with $OPT=1$ indicating that the optimal makespan is known in advance to be of unit duration and with $online$ meaning that items are given in an online fashion. 

Optimal algorithms (in the sense of stretching factor) for various numbers of available bins are still unknown; some research has been done on both lower and upper bounds for the optimal stretching factor. Interestingly, some of the best known lower and upper bounds (algorithms) for this problem were recently found through computational searches: see \cite{Gabay2017, BOHM20221, lhomme2022online} for lower bounds and \cite{Liesk} for upper bounds -- additionally \cite{ROB} improved the best known algorithm to 1.495 for large $m$, although without using computational searches. Figure~\ref{obs_results} depicts the current state of the art on this problem. Solid squares or solid circles correspond to values that were found through a computational search. Finding optimal online algorithms is still a challenge when the number of bins $m$ is greater than 2. Note that the gap between upper and lower bounds is still important.

\begin{figure}[htbp]
    \centering
    \begin{tikzpicture}
    \begin{axis}[
        xlabel={Number of bins $m$},
        xmin=1, xmax=10,
        ymin=1.3, ymax=1.55,
        xtick={2, 3, 4, 5, 6, 7, 8, 9},
        xticklabels={2, 3, 4, 5, 6, 7, 8, $m\geq 9$},
        ytick={1.3, 1.33, 1.35, 1.4, 1.45, 1.5},
        legend pos=north west,
        ymajorgrids=true,
        grid style=dashed,
    ]
    
    \addplot[
        scatter,
        point meta=explicit symbolic, scatter/classes={
        a={mark=square,red}, b={mark=square*,red}
        }]
        coordinates {
        (2,1.3333) [a]
        (3,1.375) [a]
        (4,1.393) [b]
        (5,1.410) [b]
        (6,1.429) [b]
        (7,1.455) [b]
        (8,1.462) [b]
        (9,1.495) [a]
        };
    \addplot[
        scatter,
        point meta=explicit symbolic, scatter/classes={
        c={mark=o,blue}, d={mark=*,blue}
        }]
        coordinates {
        (2,1.3333) [c]
        (3,1.3659) [d]
        (4,1.35714) [d]
        (5,1.35714) [d]
        (6,1.3636) [d]
        (7,1.3636) [d]
        (8,1.3636) [d]
        (9,1.3333) [c]
        };
        \legend{Upper bound,,Lower bound,}
    \end{axis}
    
    \end{tikzpicture}
    \caption{Best known bounds on the minimal stretching factor of deterministic algorithms for the online bin stretching problem. Solid squares and circles represent bounds that were found with the methods presented in Sections~\ref{section:lower bound method} and \ref{section:upper bound method}.}
    \label{obs_results}
\end{figure}
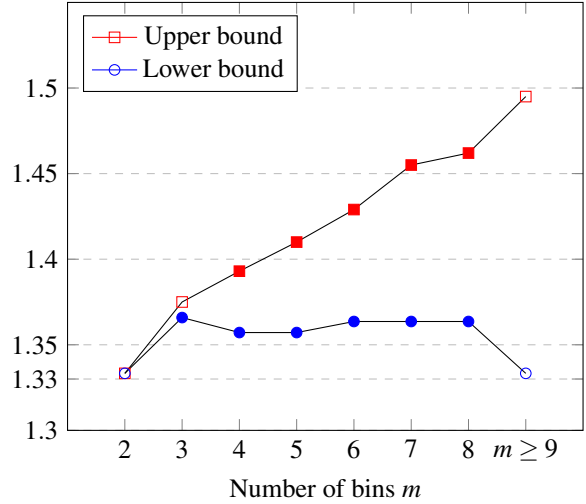

The computational method to find lower bounds may be summarized as follows: the problem is modelled as a two-player zero-sum game, where one player, the \textit{adversary}, sends items and the other player, the \textit{algorithm}, places the items into bins. The algorithm aims to minimize the load of the largest bin. By restricting the adversary to a finite set of possible item sizes, the game becomes finite -- and hence its min-max value may be computed through a standard min-max algorithm. Since the adversary player is restricted in this game, the min-max value is a lower bound of the best achievable competitive ratio of online deterministic algorithms. By increasing the set of items available to the adversary, one may hope to improve the lower bounds derived from the corresponding game; however, this also greatly expands the game size: the limiting factor of this method in obtaining better bounds becomes computation time. Figure~\ref{tree_proof_4/3} gives an example of an adversarial strategy (from \cite{AZAR200117}) that proves the lower bound $\frac{4}{3}$ when 2 bins are available -- such a strategy may be computed through a standard min-max algorithm.

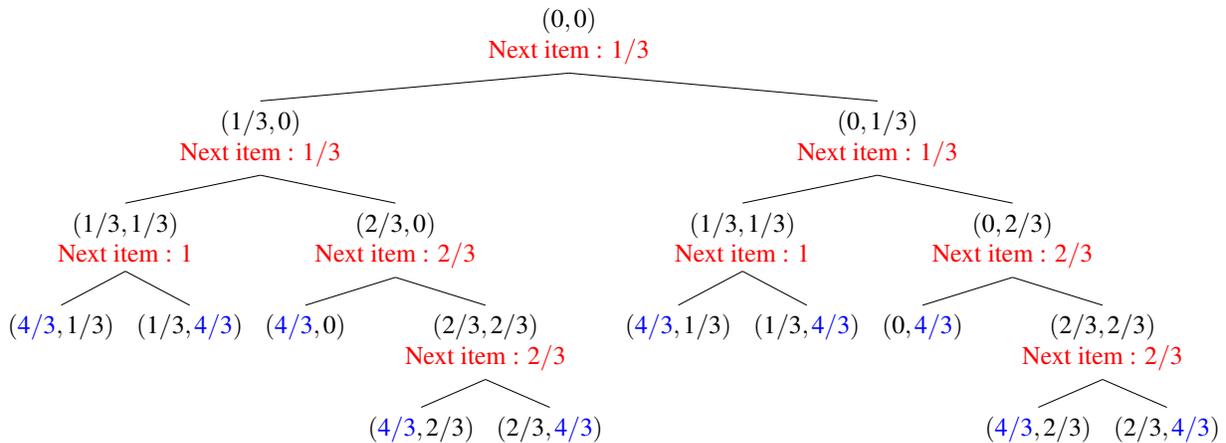
\begin{figure}[htpb]
    \centering
    \resizebox{\linewidth}{!}{
	\begin{tikzpicture}
	    \tikzset{every tree node/.style={align=center,anchor=north}}
	    \tikzset{level distance=40pt}
	    \Tree
	    [.{$(0,0)$\\{\color{red} Next item : $1/3$}}
	    [.{$(1/3, 0)$\\{\color{red} Next item : $1/3$}}
	    [.{$(1/3, 1/3)$\\{\color{red} Next item : $1$}}
	    [.{$({\color{blue} 4/3}, 1/3)$} ]
	    [.{$(1/3, {\color{blue} 4/3})$} ]
	    ]
	    [.{$(2/3, 0)$\\{\color{red} Next item : $2/3$}}
	    [.{$({\color{blue} 4/3}, 0)$} ]
	    [.{$(2/3, 2/3)$\\{\color{red} Next item : $2/3$}}
	    [.{$({\color{blue} 4/3}, 2/3)$} ]
	    [.{$(2/3, {\color{blue} 4/3})$} ]
	    ]
	    ]
	    ]
	    [.{$(0, 1/3)$\\{\color{red} Next item : $1/3$}}
	    [.{$(1/3, 1/3)$\\{\color{red} Next item : $1$}}
	    [.{$({\color{blue} 4/3}, 1/3)$} ]
	    [.{$(1/3, {\color{blue} 4/3})$} ]
	    ]
	    [.{$(0, 2/3)$\\{\color{red} Next item : $2/3$}}
	    [.{$(0, {\color{blue} 4/3})$} ]
	    [.{$(2/3, 2/3)$\\{\color{red} Next item : $2/3$}}
	    [.{$({\color{blue} 4/3}, 2/3)$} ]
	    [.{$(2/3, {\color{blue} 4/3})$} ]
	    ]
	    ]
	    ]
	    ]
	\end{tikzpicture}
	}
    \caption{Adversarial strategy showing the lower bound of $\frac{4}{3}$ for the online bin stretching problem where 2 bins are available. In parentheses: occupied volume in each bin.}
    \label{tree_proof_4/3}
\end{figure}

To construct upper bounds, \textit{i.e.}, algorithms, the same kind of technique has been proposed in \cite{Liesk}. The problem is once again modelled as a game between the adversary and the algorithm. To ensure that a strategy for the algorithm player in that game corresponds to a true online bin stretching algorithm, the algorithm strategy needs to be able to process any item sequence. Since item sizes may be any real numbers, the method proposed by \cite{Liesk} is to split possible item sizes into a finite number of intervals, and then to identify item sizes as the corresponding interval: the adversary player does not send an item of some given size but rather an item of size within some interval. In such a game, a strategy for the algorithm player is a valid online bin stretching algorithm. As such, the min-max value of that game yields an upper bound on the optimal performance of an online algorithm. 

One may then hope that increasing the quantity of intervals results in more precise algorithms and thus better upper bounds. Once again, doing so increases significantly the size of the game tree which makes computation times the limiting factor to obtain better bounds.

To summarize, the heart of the computational methods to obtain the lower and upper bounds consists of min-max searches in a game tree. In order to obtain better proofs, one may increase the possibilities of items that may be sent -- however, doing so increases exponentially the game tree size. As such, computation times are the limiting factor for obtaining better bounds. 
A very natural question is then: does increasing item possibilities yield bounds as close as possible to the true optimum value of the online bin stretching problems; in other words, do the lower and upper bound methods for the online bin stretching problem converge, and at what speed? Our paper shows that both methods do indeed converge towards the optimal performance of an online algorithm; moreover, the speed of convergence of these methods will also be discussed.

The idea to approximate the competitive ratio of online algorithms was notably developed in \cite{CRAS}, where the existence of schemes to approximate the optimal competitive ratio for a wide class of scheduling problems is proved. However, that result is theoretical, and no efficient method is proposed; to the best of our knowledge, that technique has not been applied successfully on some problem to improve bounds via computational searches. Our work hence differs in that we show the convergence of a specific computational technique, which proved most of the best-known bounds today. 

\subsection{Computational method for lower bounds}\label{section:lower bound method}
The lower bound method for the online bin stretching problem is detailed here. Let $g$ be some integer corresponding to the game \textit{granularity}, let $m$ be the number of bins available. Consider the following zero-sum, two player game, between the \textit{algorithm} and the \textit{adversary}:

The adversary sends an item of size within $\{\frac{1}{g}, \frac{2}{g}, \dots, 1\}$ satisfying the constraint that the sequence of items sent must fit into $m$ bins of size~$1$; the algorithm must then place that item into some bin of its choice. This is repeated until the adversary may no longer send an item (this occurs eventually because of the constraint that the items fit into $m$ bins of size~$1$); the score of the game is then the load of the fullest bin. 

The min-max value of that game (where the adversary aims to maximize and the algorithm to minimize) is a lower bound on the performance of any online algorithm. Hence, a min-max algorithm yields a lower bound on the optimal stretching factor for online bin stretching. 

Some works have been done in order to improve the efficiency of the min-max search \cite{Gabay2017, BOHM20221, lhomme2022online}.

\subsection{Computational method for upper bounds}\label{section:upper bound method}

To construct upper bounds, the same kind of technique as for lower bounds has been used with success in \cite{Liesk}. In the lower bound game, finding a strategy for the adversarial player also induces a proof of lower bound for the online bin stretching problem. In order to find upper bounds, a promising idea could be to find a strategy for the algorithm player. However, doing so does not yield a true online bin stretching algorithm: such a strategy is only able to process items of size $x/g$ for $x$ integer, when $g$ is the granularity parameter of the lower bound game. To circumvent this issue, \cite{Liesk} proposed to consider item \textit{classes} which correspond to intervals of possible item sizes or bin loads. A class is represented by an integer $\bar y$ and corresponds to the interval $]\frac{\bar y}{g}, \frac{\bar y+1}{g}]$. 
When placing an item of class $\bar y$ into a bin of load class $\bar b$, the load of the bin becomes the interval $]\frac{\bar y + \bar b}{g}, \frac{\bar y + \bar b + 2}{g}]$ which is too wide for being a class. To deal with this issue, the adversary is also able to ``cheat'' by choosing, once the item has been placed into a bin, if the resulting interval becomes of class $\bar y + \bar b$ or of class $\bar y + \bar b + 1$. The latter is called an \textit{overflow}.

For some granularity parameter $g\in\mathbb N^*$ and a number of bins $m$, the \textit{upper bound game} is defined as follows:

The adversary moves are constrained by the following constraints, which must always be respected:
\begin{itemize}
    \item The sequence of item classes $\bar y_1, \bar y_2, \dots, \bar y_k$ (seen as integers) fits into $m$ bins of size $g-1$. If this constraint is violated, it means that regardless of the true item sizes, it is impossible to fit the items into $m$ bins of unit size (the original online bin stretching constraint).
    \item The sum of bin load classes (as integers) must be less than or equal to $mg-1$. If this constraint is violated, it implies that the item sequence given by the adversary cannot fit into $m$ bins of unit size, as the weights of the items given is strictly greater than the sum of bin load classes.
    \item To let the game progress, if the adversary sends an item of class $\bar 0$ then an overflow must occur.
\end{itemize}

The adversary sends an item of class $\bar y \in \mathbb N$. The algorithm places the received item into a bin. The adversary may then choose whether an overflow occurs or not, \textit{i.e.}, if the class of the bin increases by $\bar y$ or by $\bar y + 1$. When the adversary is unable to send any item, the game ends and the resulting score is the upper bound of the fullest bin load, \textit{i.e.}, the score is the highest class of the load of a bin plus one. 

A strategy for the algorithm player in this game is a decision tree that decides for any incoming item class into which bin to place that item. Such a strategy is also a valid online bin stretching algorithm -- hence the min-max value of that game is an upper bound on the optimal bin stretching factor for online bin stretching.

\section{Convergence}
In all that follows, consider the number of bins $m$ fixed, and define by $\mathcal A$ the set of online bin stretching algorithms and $\mathcal I$ the set of possible input item sequences. Let $A\in\mathcal A$ be an online algorithm. If $A(I)$ denotes the load of the fullest bin after packing the items from the sequence $I\in\mathcal I$ according to $A$, then the stretching factor $r(A)$ of $A$ is defined as:
$$r(A) = \sup_{I \in \mathcal I} A(I)$$
Naturally, the aim is to design algorithms as efficient as possible; thus we are interested in the value of the optimum stretching factor:
$$v^* = \inf_{A\in\mathcal A} r(A)$$ 

While it seems fairly reasonable that the lower bound method converges towards $v^*$, it is much less obvious for the upper bound method. The adversary player is able to ``cheat'' via overflows in the upper bound game. While increasing the granularity reduces the power of one single overflow, it also increases the length of the game, hence the adversary is able to send more items with overflows. It is unclear whether the adversary has too much power in the upper bound game to derive arbitrarily good upper bounds or not.

The following will show that the upper bound method does in fact converge towards $v^*$. The proof will also give some insight on the convergence speed, and may be used to prove the convergence of the lower bound method as well as the same kind of result on its speed.

\begin{theorem}\label{thm:upper bound convergence}
    
The upper bound method described in Section~\ref{section:upper bound method} does converge towards the optimum stretching factor. Formally, if $(u_g)_{g\in\mathbb N^*}$ denotes the sequence of upper bounds found for each granularity, then: $$\lim_{g\to +\infty} u_g = v^*$$

\end{theorem}

\textit{Idea of the proof}: From $A$ an online bin stretching algorithm and $g\in \mathbb N^*$, we construct $A^g$ an algorithm strategy for the upper bound game with granularity $g$, so that the stretching factor of $A^g$ is close to the stretching factor of $A$:
$$| r(A) - r(A^g) | = \underset{g\to +\infty}{o}(1)$$  

Since the min-max value $u_g$ obtained from the upper bound game with granularity $g$ is less than the performance of any algorithm strategy in that game, the following holds:
$$\forall A\in\mathcal A,\;v^*\leq u_g \leq r(A^g) \leq r(A) + \underset{g\to +\infty}{o}(1)$$
Selecting some algorithm $A$ close to the optimum stretching factor thus implies:
$$v^* \leq u_g \leq \inf_{A\in\mathcal A} r(A) + \underset{g\to +\infty}{o}(1) = v^* + \underset{g\to +\infty}{o}(1)$$
Hence:
$$\lim_{g\to +\infty} u_g = v^* $$

\textbf{Definitions and notations:}

\begin{itemize}
    \item An online bin stretching algorithm $A$ will be defined as a sequence of functions $(A_1, A_2, \dots)$, such that the $i$-th function outputs the $i$-th move. $A_i$ takes as input a finite sequence of items of length $i$, and outputs the index of the bin into which to place the $i$-th item. 
    \item An algorithm strategy $A^g$ for the upper bound game will be defined as a sequence of functions $(A^g_1, A^g_2, \dots)$. $A^g_i$ takes as input a finite sequence of item classes of length $i$: $(\bar y_1, \dots, \bar y_i)$ and a finite sequence of binary values of length $i-1$ corresponding to if an overflow occurred in one of the previous moves: $(o_1, \dots, o_{i-1})$, where $o_j = 1$ if and only if an overflow occurred after placing the $j$-th item into a bin. $A^g_i$ outputs the index of the bin into which to place the $i$-th item.
    \item The concatenation of finite sequences will be denoted by $\oplus$: $(1, 2, 3) \oplus (4, 5) = (1, 2, 3, 4, 5)$. The empty sequence is denoted by $\emptyset$.
\end{itemize}

\begin{proof}

    Let $g\in\mathbb N^*$ and let $A$ be an online bin stretching algorithm where the offline bin size\footnote{Online bin stretching is defined with the constraint that items fit into $m$ bins of unit size -- to deal with integers rather than fractions for item sizes, the problem is scaled up to bins of size $g$ or $g'$ without loss of generality.} is fixed to be $g' = g(1 + \frac{m}{\sqrt g} + \frac{1}{g}) + \sqrt{g(1 + \frac{m}{\sqrt g} + \frac{1}{g})} = g + \underset{g\to +\infty}{o}(g)$ -- the reason for that will be made clear later on. We aim to construct $A^g$ which behaves like $A$ in most cases. 

    We construct inductively $A^g = (A^g_1, A^g_2, \dots)$, as well as variables $\delta_{i} = (\delta_{i, 1}, \dots, \delta_{i, m})$ and functions $M_i$.
    
    \textbf{Idea behind these variables:} The constructed algorithm $A^g$ tries to behave like $A$. $A^g$ will keep track of the solutions proposed by $A$ -- both algorithms are executed together. As the adversary does have additional power (\textit{moves}) in the upper bound game via overflows, it is not possible for $A^g$ to directly ``copy'' the moves from $A$. For example, upon receiving an item of class 1, \textit{i.e.}, of size between 1 and 2, we could ``ask'' the algorithm $A$ what decision it would take if the item was of size 2. However, upon following the decision of $A$ in our constructed algorithm, no overflow occurs -- the item represented by the class $1$ was perhaps very close to size 1. Errors may hence accumulate by following blindly algorithm $A$. 
    
    The algorithm $A^g$ will always assume that an overflow occurs. If it does not, then a game state may be reached that the real algorithm $A$ may not usually reach. As such, we keep track of the difference between what $A$ believes is the current load of each bin and the actual load class of each bin with variables $\delta$: $\delta_{i, j}$ corresponds to the difference between the load of the class (as a number) of the $j$-th bin in the constructed algorithm $A^g$ and the load of the $j$-th bin in algorithm $A$ just before receiving the $i$-th item.

    Initially, $\delta_{1, j} = 0 \;\forall j\in \{1, \dots, m\}$ : there is no gap between $A$ and $A^g$ prior to receiving items. The idea behind the construction that follows is to place small items into bins for which $\delta$ is negative without taking into account the decision of algorithm $A$. Doing so allows $A^g$ to catch up with $A$. Functions $M_k$ ($M$ stands for \textit{memory}) will keep track of the items for which $A^g$ behaves like $A$: for some input $((\bar y_1, \bar y_2, \dots, \bar y_i), (o_1, \dots, o_{i-1}))$ of the algorithm $A^g$ at step $i$, $M_i((\bar y_1, \bar y_2, \dots, \bar y_{i-1}), (o_1, \dots, o_{i-1}))$ will correspond to the sequence of item classes for which $A^g$ behaved like $A$ so far; initially, $M_1((\emptyset, \emptyset)) = \emptyset$. 
    
    The algorithm $A^g = (A^g_i)_{i\in\mathbb N^*}$ is defined alongside $M_i$ by induction on $i$.

    Let $(\bar I\oplus \bar y_i, o) = ((\bar y_1, \dots, \bar y_i), (o_1, \dots, o_{i-1}))$ be some input for algorithm $A^g$ at step $i$; let $k = |M(\bar I, o)|$ be the number of items for which $A^g$ behaved like $A$ so far. Define:

    $$
    A^g_i((\bar I\oplus \bar y_i, o)) = \begin{cases}
        A_{k+1}(M_{i}(\bar I, o)\oplus  (\bar y_i + 1)),\;\text{ if } \bar y_i \geq \sqrt{g}\\(\textit{large item, behave like }A\textit{ would on the item }\bar y_i +1)\\

        A_{k+1}(M_{i}(\bar I, o) \oplus  (\bar y_i + 1)),\; \text{ if } \bar y_i < \sqrt{g} \text{ and } \forall j\; \delta_{i, j}\geq 0\\ (\textit{small item and no gap exists, behave like }A)\\

        \min\limits_{j\in\{1, \dots, m\}} j \;|\; \delta_{i, j} < 0\; \text{ otherwise}\\
        (\textit{small item, bridge an existing gap})
    \end{cases}
    $$

    Then, the adversary decides if an overflow happens or not in the bin chosen by the algorithm, and fixes the binary value $o_i$.

    The term $M_i(\bar I\oplus \bar y_i, o\oplus o_i)$ is then defined accordingly as:
    $$
    M_{i+1}(\bar I\oplus \bar y_i, o\oplus o_i) = \begin{cases}
        M_{i}(\bar I, o)\oplus  (\bar y_i + 1),\;\text{ in the first or second case of the previous construction}\\
        M_i(\bar I, o)\; \text{ otherwise}
    \end{cases}
    $$

    Let $j$ be the bin chosen by $A^g_i$. In cases 1 and 2, define $\delta_{i+1, j} = \delta_{i, j} - 1 $ if there is no overflow in the chosen bin, otherwise $\delta$ stays the same for the next iteration.

    In case 3, change the corresponding $\delta_{i,j}$ for the next iteration by adding $\bar y_i + o_i$. For any other bin $j'\neq j$, $\delta_{i+1, j'} = \delta_{i, j'}$. 

    A graphical example of the execution of algorithm $A^g$ may be found in Figure~\ref{fig:execution ag}. 
\setlength{\fboxsep}{0pt}
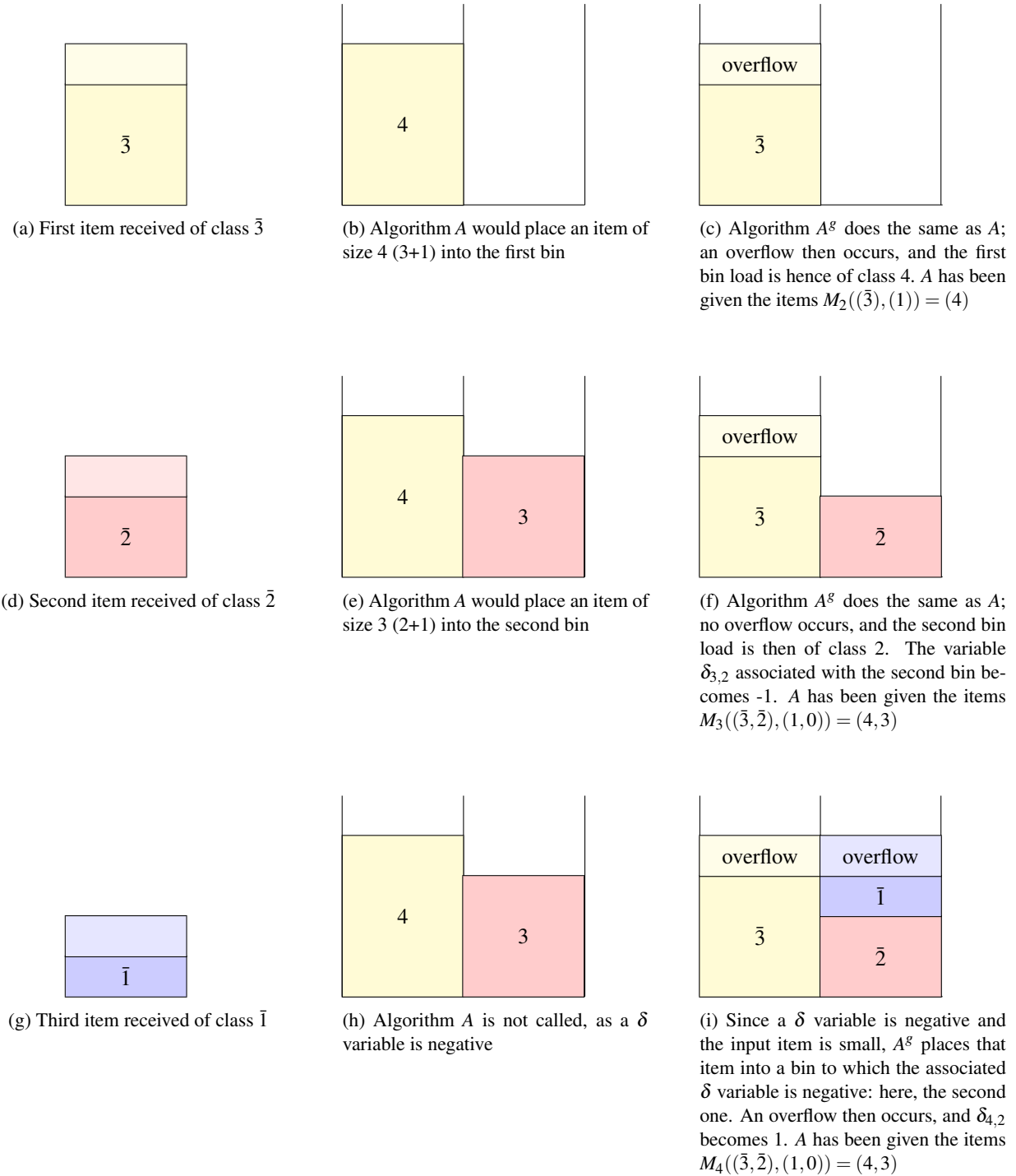
\begin{figure}[htbp]\unitlength = .65cm
     \centering
     \begin{subfigure}[t]{0.3\textwidth}

         \begin{picture}(2,6.5)

		\put(2,0){\colorbox{yellow!20}{\framebox(3,3){$\bar 3$}}}
        \put(2,3){\colorbox{yellow!10}{\framebox(3,1){}}}
	    \end{picture}
         \caption{First item received of class~$\bar 3$}
         \label{firstitem}
     \end{subfigure}
     \hfill
     \begin{subfigure}[t]{0.3\textwidth}
         \begin{picture}(2,6.5)
		\put(-.01,0){\line(0,1){5}}
		\put(3.01,0){\line(0,1){5}}
		\put(6.03,0){\line(0,1){5}}
		\put(-.01,0){\line(1,0){6.04}}

		\put(0,0){\colorbox{yellow!20}{\framebox(3,4){$4$}}}
	    \end{picture}
         \caption{Algorithm $A$ would place an item of size $4$ (3+1) into the first bin}
         \label{seconditem}
     \end{subfigure}
     \hfill
     \begin{subfigure}[t]{0.3\textwidth}
         \begin{picture}(2,6.5)
		\put(0,0){\line(0,1){5}}
		\put(3,0){\line(0,1){5}}
		\put(6,0){\line(0,1){5}}
		\put(0,0){\line(1,0){6.04}}

		\put(0,0){\colorbox{yellow!20}{\framebox(3,3){$\bar 3$}}}
        \put(0,3){\colorbox{yellow!10}{\framebox(3,1){overflow}}}
	    \end{picture}
         \caption{Algorithm $A^g$ does the same as $A$; an overflow then occurs, and the first bin load is hence of class $4$. $A$ has been given the items $M_2((\bar 3), (1)) = (4)$}
         \label{thirditem}
     \end{subfigure}
        
        \begin{subfigure}[t]{0.3\textwidth}

            \begin{picture}(2,6.5)
   
           \put(2,0){\colorbox{red!20}{\framebox(3,2){$\bar 2$}}}
           \put(2,2){\colorbox{red!10}{\framebox(3,1){}}}
           \end{picture}
            \caption{Second item received of class~$\bar 2$}
            \label{firstitem2}
        \end{subfigure}
        \hfill
        \begin{subfigure}[t]{0.3\textwidth}
            \begin{picture}(2,6.5)
           \put(-.01,0){\line(0,1){5}}
           \put(3.01,0){\line(0,1){5}}
           \put(6.03,0){\line(0,1){5}}
           \put(-.01,0){\line(1,0){6.04}}

           \put(0,0){\colorbox{yellow!20}{\framebox(3,4){$4$}}}
           \put(3,0){\colorbox{red!20}{\framebox(3,3){$3$}}}
           \end{picture}
            \caption{Algorithm $A$ would place an item of size $3$ (2+1) into the second bin}
            \label{seconditem2}
        \end{subfigure}
        \hfill
        \begin{subfigure}[t]{0.3\textwidth}
            \begin{picture}(2,6.5)
           \put(0,0){\line(0,1){5}}
           \put(3,0){\line(0,1){5}}
           \put(6,0){\line(0,1){5}}
           \put(0,0){\line(1,0){6.04}}

           \put(0,0){\colorbox{yellow!20}{\framebox(3,3){$\bar 3$}}}
           \put(0,3){\colorbox{yellow!10}{\framebox(3,1){overflow}}}
           \put(3,0){\colorbox{red!20}{\framebox(3,2){$\bar 2$}}}

           \end{picture}
            \caption{Algorithm $A^g$ does the same as $A$; no overflow occurs, and the second bin load is then of class $2$. The variable $\delta_{3, 2}$ associated with the second bin becomes -1. $A$ has been given the items $M_3((\bar 3, \bar 2), (1, 0)) = (4, 3)$}
            \label{thirditem2}
        \end{subfigure}


        \begin{subfigure}[t]{0.3\textwidth}

            \begin{picture}(2,6.5)
   
           \put(2,0){\colorbox{blue!20}{\framebox(3,1){$\bar 1$}}}
           \put(2,1){\colorbox{blue!10}{\framebox(3,1){}}}
           \end{picture}
            \caption{Third item received of class~$\bar 1$}
            \label{firstitem3}
        \end{subfigure}
        \hfill
        \begin{subfigure}[t]{0.3\textwidth}

            \begin{picture}(2,6.5)
                \put(-.01,0){\line(0,1){5}}
                \put(3.01,0){\line(0,1){5}}
                \put(6.03,0){\line(0,1){5}}
                \put(-.01,0){\line(1,0){6.04}}

                \put(0,0){\colorbox{yellow!20}{\framebox(3,4){$4$}}}
                \put(3,0){\colorbox{red!20}{\framebox(3,3){$3$}}}
                \end{picture}

            \caption{Algorithm $A$ is not called, as a $\delta$ variable is negative}
            \label{second item3}
        \end{subfigure}
        \hfill
        \begin{subfigure}[t]{0.3\textwidth}
            \begin{picture}(2,6.5)
           \put(0,0){\line(0,1){5}}
           \put(3,0){\line(0,1){5}}
           \put(6,0){\line(0,1){5}}
           \put(0,0){\line(1,0){6.04}}

           \put(0,0){\colorbox{yellow!20}{\framebox(3,3){$\bar 3$}}}
           \put(0,3){\colorbox{yellow!10}{\framebox(3,1){overflow}}}
           \put(3,0){\colorbox{red!20}{\framebox(3,2){$\bar 2$}}}
           \put(3,2){\colorbox{blue!20}{\framebox(3,1){$\bar 1$}}}
           \put(3,3){\colorbox{blue!10}{\framebox(3,1){overflow}}}
           
           \end{picture}
            \caption{Since a $\delta$ variable is negative and the input item is small, $A^g$ places that item into a bin to which the associated $\delta$ variable is negative: here, the second one. An overflow then occurs, and $\delta_{4,2}$ becomes 1. $A$ has been given the items $M_4((\bar 3, \bar 2), (1, 0)) = (4, 3)$}
            \label{thirditem3}
        \end{subfigure}
        \caption{Example of an execution of algorithm $A^g$ with $m=2$ bins and $g \geq 4$}
        \label{fig:execution ag}
\end{figure}

    One may remark that the variable $\delta$ of a bin may only decrease by $1$ in the first and second cases of the algorithm if no overflow occurs. A variable may only increase if it's negative and may only increase by $\sqrt g + 1$. Since there can be at most $m\sqrt g$ large items, the following invariant holds:
    \begin{equation}\label{eq:bound on delta}
        \forall i\;\forall j,\;-m\sqrt g-1\leq \delta_{i, j} \leq \sqrt{g}
    \end{equation}

    As such, for any input sequence, the absolute difference between the load of a bin for algorithm $A^g$ and the load of that bin for algorithm $A$ is at most $m\sqrt g + 1$; by definition, the fullest bin of algorithm $A$ has a load less than $r(A)g'$. Thus, the stretching factor of $r(A^g)$ may be bounded:
    $$r(A^g) g \leq r(A) g' + (m\sqrt g + 1) +1$$
    
    \begin{equation}\label{eq:diff between loads is bounded}
        r(A^g) \leq \frac{r(A)  g' + (m\sqrt g + 1) +1}{g} = r(A) + \underset{g\to +\infty}{o}(1)
    \end{equation}
    The last term $+1$ in the above equations comes from the fact that bin loads are defined as intervals, thus the $+1$ is to consider an upper bound. As established in the idea of the proof, this shows that the upper bounds converge towards $v^*$; the previous inequality is valid for any algorithm $A$ and since $r(A^g)\geq u_g$, the following holds:
    $$v^*\leq u_g \leq \frac{v^*g' + m\sqrt g + 2}{g}$$
    This shows that upper bounds do converge towards $v^*$.

    We now prove that $A^g$ is well-defined. This is needed because the constructed algorithm $A^g$, when copying algorithm $A$'s behavior, sends slightly larger items than what they could be to algorithm $A$. This sequence of slightly larger items could be a non-valid input for $A$. The following shows that this is not the case if the offline bin size for algorithm $A$ is $g'$, which is slightly greater than $g$.

    First, consider the following lemma:
    
    \begin{lemma}\label{lemma:camion crous}
        Let $h\in\mathbb N^*$, let $I=(y_1, \dots, y_n)\in \mathbb{N}^n$ such that:
        \begin{enumerate}
            \item $(y_1, \dots, y_n)$ \text fit into $m$ bins of size $h$
            \item $n + \sum_{i=1}^n y_i  < mh$ 
        \end{enumerate}
        Then, the sequence of items  $(y_1 + 1, \dots, y_n + 1)$ fits into $m$ bins of size $h + \sqrt{h}$
    \end{lemma}
    
    \begin{proof}
        \textit{(Lemma~\ref{lemma:camion crous})}

        To prove this lemma, start with a packing of the items $I$ such that no bin has a load over $h$ -- such a packing exists because of the first condition. Then, all item sizes increase by 1; some bins may hence have a load greater than $h+\sqrt h$. We will hence repair the packing by moving some items between bins. First, define the following two categories of bins based on their load: 
        \begin{itemize}
            \item if the load of a bin is strictly greater than $h+\sqrt{h}$, the bin is defined to be \textbf{critical}
            \item if the load is strictly smaller than $h$, the bin is \textbf{safe}
        \end{itemize}
        We aim to have no critical bins. First, observe that there are at least $\sqrt{h}$ items in a critical bin, hence one item has size at most $\frac{h+\sqrt{h}}{{\sqrt{h}}}=\sqrt{h}+1$. Observe also that because of condition 2, a safe bin always exists.

        Such an item may be moved to a safe bin, which cannot become critical. The operation may be repeated until no bin is critical. As the number of items in critical bins is strictly decreasing, eventually no critical bin exists. 
    \end{proof}

    With this lemma, we may now prove that the items given to $A$ are valid. Let $((y_1, \dots, y_n), (o_1, \dots, o_n))$ be an input for $A^g$. After execution of the algorithm on this input, $A^g$ also simulated the behavior of $A$ on some items. Remark that Equation~(\ref{eq:bound on delta}) implies that the difference in norm 1 between the loads of bins of the algorithm $A^g$ and $A$ (as vectors) is at most $m(m\sqrt g + 1)$; since $\sum \bar y_i + o_i < mg$ by constraint of the upper bound game, the items given to algorithm $A$ sum to at most $m(g + m\sqrt g + 1) = mg ( 1 + \frac{m}{\sqrt g} + \frac{1}{g})$. By defining $h = g ( 1 + \frac{m}{\sqrt g} + \frac{1}{g}) $, Lemma~\ref{lemma:camion crous} may be applied: items given to algorithm $A$ are thus guaranteed to fit into $m$ bins of size $g' = h+\sqrt h = g ( 1 + \frac{m}{\sqrt g} + \frac{1}{g}) + \sqrt{g ( 1 + \frac{m}{\sqrt g} + \frac{1}{g}) } = g + o(g)$.

    In conclusion, items given to $A$ do fit into $m$ bins of size $g'$, thus $A^g$ is well-defined, which concludes the proof of Theorem~\ref{thm:upper bound convergence}.

\end{proof}

It may be observed that this proof went a step further than just showing the convergence, as it showed the inequality:
\begin{equation}\label{eq:ineq in the proof between v and ug}
    v^*\leq u_g \leq \frac{v^*g' + m\sqrt g + 2}{g}
\end{equation}
Equation~(\ref{eq:ineq in the proof between v and ug}) may be written as:
\begin{corollary}
    In finding the best possible upper bound $u_g$ for some granularity $g\in\mathbb N^*$, one also obtains the lower bound:
$$ v^*\geq \left(u_g - \frac{m\sqrt g + 2}{g}\right)\frac{g}{g'}$$
\end{corollary}

In other words, this corollary gives information on the distance between the upper bound and the optimal stretching factor $v^*$.

The proof of Theorem~\ref{thm:upper bound convergence} constructs an algorithm $A^g$ for the upper bound game from an online bin stretching algorithm $A$. But that construction only made use of $A$ for items of the form $\frac{x}{g'}$ with $x$ integer. For some granularity $g'$, finding the best lower bound $\ell_{g'}$ at this granularity also implies the existence of an algorithm strategy in the lower bound game of this granularity with the same performance as the lower bound, since the lower bound game is finite. That algorithm strategy is only able to process items of the form $\frac{x}{g'}$ -- hence this strategy may play the role of $A$ in the proof of Theorem~\ref{thm:upper bound convergence}! Upon finding the min-max value $\ell_{g'}$ of the lower bound game for some granularity $g' = \left\lceil g ( 1 + \frac{m}{\sqrt g} + \frac{1}{g}) + \sqrt{g ( 1 + \frac{m}{\sqrt g} + \frac{1}{g}) }\;\right\rceil$, one may apply the construction in the proof of Theorem~\ref{thm:upper bound convergence} that constructs an algorithm strategy in the upper bound game, \textit{i.e.}, an algorithm with performance close to $\ell_{g'}$. In equation~(\ref{eq:diff between loads is bounded}) in the proof of Theorem~\ref{thm:upper bound convergence}, the term $r(A)$ may be replaced by $\ell_{g'}$ to obtain:

\begin{theorem}
    The following inequality holds:
    $$\ell_{g'} \leq v^* \leq \left( \ell_{g'} + \frac{m\sqrt g + 2}{g'}\right)\frac{g'}{g} $$
    This inequality also implies that the sequence of lower bounds $\ell_{g'}$ found at each granularity ${g'}$ converges towards $v^*$. 
\end{theorem}

\textbf{Remark:} the given inequalities are too weak to deduce anything from the bounds known today. For instance, for 4 bins, with granularity $g=22$ the best algorithm strategy has a stretching factor of at most $31/22\approx 1.409$ (see \cite{Liesk}); the lower bound implied by $u_{22} = 31/22$ is $v^*\geq 0.212$... which is not even close to being a useful result -- for comparison, the best known lower bound for 4 bins is $19/14 \approx 1.3571$. Despite this, our result still holds practical interest: showing the convergence of computational methods gives a theoretical foundation for people aiming to push further computational methods in order to obtain even stronger bounds.

\section{Conclusion}
Computational methods for online problems hold much promise for constructing better bounds and algorithms with performance guarantees. In online bin stretching, most of the best bounds and algorithms known were found through such methods. This paper gives a theoretical foundation for such methods by showing their convergence towards the optimal online performance -- as well as giving a bound on the gap to that optimal. 
Beyond the purely theoretical interest of these results, an interesting perspective would be to use those bounds to accelerate computational searches in practice.

\bibliography{refs.bib}
\bibliographystyle{plain}

\end{document}